\documentclass{birkjour}
\usepackage{amsmath}
\usepackage{amsfonts}
\usepackage{latexsym}
\usepackage{amssymb}
\usepackage{color}

\chardef\bslash=`\\ 

\makeatletter
\def\verbatim{\interlinepenalty\@M \@verbatim
  \leftskip\@totalleftmargin\advance\leftskip2pc
  \frenchspacing\@vobeyspaces \@xverbatim}
\makeatother
\newtheorem{thm}{Theorem}[section]
\newtheorem{cor}[thm]{Corollary}

\newtheorem{rem}[thm]{Remark} 

\newtheorem{ques}[thm]{Question}

\numberwithin{equation}{section}

\newcommand{\ZZ}{{\mathbb Z}}
\newcommand{\RR}{{\mathbb R}}

\newcommand{\R}{ {\Bbb R}}
\newcommand{\Z} {{\Bbb Z}}

\begin{document}

\title[Differential Subalgebras and Norm-Controlled Inversion] {Differential subalgebras and norm-controlled inversion}

\author{Chang Eon Shin}
\address{Department of Mathematics, Sogang University, Seoul 121-742, Korea.
Email: shinc@sogang.ac.kr }

\author{Qiyu Sun}

\address{Department of Mathematics, University of Central Florida, Orlando, FL 32816, USA.  Email: qiyu.sun@ucf.edu
}

\thanks{The project is partially supported by   the Basic
Science Research Program through the National Research Foundation of Korea (NRF) funded by the Ministry of
Education, Science and Technology (NRF-2016R1D1A1B03930571) and
the National Science Foundation (DMS-1816313) }




\dedicatory{In  memory of Ronald G. Douglas}

\begin{abstract} In this paper, we consider the
norm-controlled inversion for  differential  $*$-subalgebras of a symmetric $*$-algebra with common identity and involution.
\end{abstract}

\maketitle

\section{Introduction}

In \cite[Lemma IIe]{wiener32}, it states that
``{\em If $f(x)$ is a function with an absolutely convergent Fourier series, which nowhere vanishes for real arguments, $1/f(x)$ has an absolutely convergent Fourier series.}"
The above statement is now known as the  classical Wiener's lemma.

We say that a Banach space ${\mathcal A}$  with norm $\|\cdot\|_{\mathcal A}$ 
is  a {\em  Banach algebra} if  
it has operation of multiplications possessing the usual algebraic properties, and
\begin{equation}\label{banachalgebra.def}
\|AB\|_{\mathcal A}\le   K \|A\|_{\mathcal A}\|B\|_{\mathcal A}\ \ {\rm for\ all} \ A, B\in {\mathcal A},
\end{equation}
where $K$ is a positive constant. 
Given two Banach algebras ${\mathcal A}$ and ${\mathcal B}$ such that ${\mathcal A}$ is a Banach subalgebra of ${\mathcal B}$,
we say that
 ${\mathcal A}$ is {\it inverse-closed} in ${\mathcal B}$ if $A\in {\mathcal A}$ and
$A^{-1} \in {\mathcal B}$ implies $A^{-1}\in {\mathcal A}$.  Inverse-closedness is also known as  spectral invariance, Wiener pair, local subalgebra, etc  
 \cite{douglasbook, gelfandbook, Naimarkbook, takesaki}.
Let  ${\mathcal C}$ be the algebra of all periodic continuous functions  under multiplication, and
${\mathcal W}$  be its Banach subalgebra of all periodic functions with  absolutely convergent Fourier series,
\begin{equation}\label{Wieneralgebra.def}
{\mathcal W}=\Big\{f(x)=\sum_{n\in \ZZ} \hat f(n) e^{inx},\ \
 \|f\|_{\mathcal W}:=\sum_{n\in \ZZ} |\hat f(n)|<\infty\Big\}.\end{equation}
Then the classical Wiener's lemma can be  reformulated as that  ${\mathcal W}$ is an inverse-closed subalgebra of ${\mathcal C}$.
 Due to the above interpretation, we  also call the inverse-closed property for a Banach subalgebra ${\mathcal A}$ as Wiener's lemma for that subalgebra.
 Wiener's lemma for Banach algebras of infinite matrices and integral operators with certain off-diagonal decay
 can be informally interpreted as localization preservation under inversion.
Such a localization preservation is of great importance in  applied harmonic analysis, numerical analysis,  optimization
and many mathematical and engineering fields  \cite{akramgrochenigsiam, chengsun19, christensen05, grochenigr10,  ksw13, sunsiam06}.
  The readers may refer to the survey papers \cite{grochenig10, Krishtal11, shinsun13}, the recent  papers \cite{fangshinsun20, samei19, shinsun19} and  references therein for  historical remarks and recent advances.

Given an element $A$ in a Banach algebra ${\mathcal A}$ with the identity $I$, we define its  {\em spectral set} $\sigma_{\mathcal{A}}(A)$
and  {\em spectral radius} $\rho_{\mathcal{A}}(A)$  by
$$
\sigma_{\mathcal{A}}(A):=\big\{\lambda \in \mathbb{C} : \lambda I -A \text{ is not invertible in }
\mathcal{A} \big\}
$$
and
$$
\rho_\mathcal{A}(A) := \max \big\{ |\lambda| :\lambda  \in \sigma_{\mathcal A}(A)\big\}
$$
respectively.
 Let ${\mathcal A}$  and
${\mathcal B}$ be Banach algebras with common identity $I$ and   ${\mathcal A}$ be a Banach subalgebra of ${\mathcal B}$.
Then an equivalent condition for the inverse-closedness of
${\mathcal A}$  in ${\mathcal B}$ is that the spectral set of any
$A\in {\mathcal A}$ in  Banach algebras ${\mathcal A}$ and ${\mathcal B}$ are the same, i.e.,
$$
\sigma_{\mathcal A}(A)=\sigma_{\mathcal B}(A).
$$
By the above equivalence, a necessary condition for the inverse-closedness of
${\mathcal A}$  in ${\mathcal B}$ is that
the spectral radius of any
$A\in {\mathcal A}$ in the Banach algebras ${\mathcal A}$ and ${\mathcal B}$ are the same, i.e.,
\begin{equation} \label{spectralradius}
\rho_\mathcal{A}(A) =\rho_\mathcal{B}(A).
\end{equation}
The above necessary condition is shown  by Hulanicki \cite{hulanicki} to be sufficient if we further assume that
$\mathcal{A}$  and $\mathcal{B}$ are $*$-algebras with common identity and involution, and that
 $\mathcal{B}$ is symmetric.
  Here we say that
 a Banach algebra $\mathcal B$ is   a $*$-algebra if
there is a  continuous  linear {\em  involution $*$} on  $\mathcal {B}$
with the properties that
\begin{equation*}
(AB)^* = B^* A^*\  \text{ and }\  A^{**} = A\  \text{ for all }  A,  B \in
{\mathcal B},
\end{equation*}
and that a $*$-algebra ${\mathcal B}$ is  {\em symmetric} if
$$\sigma_{\mathcal {A}} (A^* A) \subset [0,\infty )\ \ {\rm for \ all} \  A\in {\mathcal B}.$$
The spectral radii approach \eqref{spectralradius}, known as  the Hulanicki's spectral method,
  has been used to establish the inverse-closedness of symmetric $*$-algebras \cite{branden, gkI, grochenigklotz10,  sunca11,  suntams07, suncasp05},
  however the above approach does not provide a norm estimate for the inversion, which is crucial for  many mathematical and engineering  applications.

To consider norm estimate for the inversion, we recall the concept of norm-controlled inversion   of a Banach subalgebra  ${\mathcal A}$ of a symmetric  $*$-algebra  ${\mathcal B}$, which was initiated by Nikolski \cite{nikolski99} and coined by Gr\"ochenig and Klotz \cite{gkI}. Here
we say that
 a Banach subalgebra ${\mathcal A}$ of ${\mathcal B}$  admits {\em
norm-controlled inversion} in ${\mathcal B}$ if there exists a continuous function $h$ from
$[0, \infty)\times [0, \infty)$ to $[0, \infty)$
 such that
\begin{equation}\label{normcontrol}
\|A^{-1}\|_{\mathcal A}\le h\big(\|A\|_{\mathcal A}, \|A^{-1}\|_{\mathcal B}\big)
\end{equation}
for all $A\in {\mathcal A}$ being invertible in ${\mathcal B}$
\cite{gkII, gkI,  samei19, shinsun19}.

The norm-controlled inversion is  a strong version of Wiener's lemma.
The classical Banach algebra ${\mathcal W}$   in \eqref{Wieneralgebra.def}  is inverse-closed in
the algebra
 ${\mathcal C}$ of  all periodic continuous functions
  \cite{wiener32},  however it does not have norm-controlled inversion in
 ${\mathcal C}$ \cite{belinskiijfaa97, nikolski99}.
To establish Wiener's lemma, there are several methods, including
 the Wiener's localization \cite{wiener32}, the Gelfand's technique
\cite{gelfandbook},   the Brandenburg's trick \cite{branden},  the Hulanicki's spectral method \cite{hulanicki}, the Jaffard's  boot-strap argument \cite{jaffard90},
the derivation technique \cite{grochenigklotz10}, 
and the Sj\"ostrand's commutator estimates \cite{shinsun19, sjostrand94}.
 In this  paper, we will use the Brandenburg's trick
 to establish norm-controlled inversion of
 a differential $*$-subalgebra  ${\mathcal A}$ of a symmetric  $*$-algebra  ${\mathcal B}$.

This introduction article is organized as follows. In Section \ref{differentialalgebra.section}, we recall the concept of
differential subalgebras and present some differential subalgebras of  infinite matrices with polynomial off-diagonal decay.
In Section \ref{generalizedDS.section}, we introduce  the concept of generalized differential  subalgebras and
present some generalized differential subalgebras of integral operators with kernels being  H\"older continuous and having polynomial off-diagonal decay.
In Section \ref{normcontrolledinversion.section},  we use the Brandenburg's trick  to establish norm-controlled inversion
  of   a differential $*$-subalgebra of a symmetric  $*$-algebra, and we conclude the section with two remarks on
 the  norm-controlled inversion with the norm control function bounded by a polynomial and
  the norm-controlled inversion of nonsymmetric  Banach algebras.

\section{Differential Subalgebras}\label{differentialalgebra.section}

 Let ${\mathcal A}$  and
${\mathcal B}$ be Banach algebras such that ${\mathcal A}$ is a Banach subalgebra of  ${\mathcal B}$.
We say that
${\mathcal A}$ is  a {\em differential subalgebra of order $\theta\in (0, 1]$} in ${\mathcal B}$  
if there exists a positive constant $D_0:=D_0({\mathcal A}, {\mathcal B}, \theta)$ such that
\begin{equation}\label{differentialnorm.def}
\|AB\|_{\mathcal A}\le D_0\|A\|_{\mathcal A} \|B\|_{\mathcal A} \Big (\Big(\frac{\|A\|_{\mathcal B}}{\|A\|_{\mathcal A}}\Big)^\theta +
\Big(\frac{\|B\|_{\mathcal B}}{\|B\|_{\mathcal A}}\Big)^\theta
\Big)
\quad {\rm for \ all} \  A, B \in {\mathcal A}.
\end{equation}
The concept of   differential subalgebras of order $\theta$  was introduced in \cite{blackadarcuntz91, kissin94, rieffel10} for $\theta=1$ and
\cite{christ88, gkI,     shinsun19} for $\theta\in (0, 1)$.
 We also refer the reader
to \cite{barnes87, fangshinsun13, gkII, gkI,   grochenigklotz10,   jaffard90, rssun12, samei19, sunca11, sunacha08, suntams07, suncasp05} for various differential subalgebras
of  infinite  matrices, convolution operators, and  integral operators with certain off-diagonal decay.

 For $\theta=1$, the  requirement
\eqref{differentialnorm.def} can be reformulated as
\begin{equation}
\|AB\|_{\mathcal A}\le D_0\|A\|_{\mathcal A} \|B\|_{\mathcal B}+ D_0 \|A\|_{\mathcal B} \|B\|_{\mathcal A}
\quad {\rm for \ all} \  A, B \in {\mathcal A}. \end{equation}
So the  norm $\|\cdot\|_{\mathcal A}$  satisfying \eqref{differentialnorm.def} 
is also referred  as a Leibniz norm on ${\mathcal A}$.

   Let $C[a, b]$ be the space of  all continuous functions on the interval $[a, b]$  with its  norm defined by
$$\|f\|_{C[a, b]}=\sup_{t\in [a, b]} |f(t)|, \ \  f\in C[a, b],$$
 and $C^k[a, b], k\ge 1$, be the space of all continuously  differentiable functions on the interval $[a, b]$ up to order $k$
with its norm defined by
$$\|h\|_{C^k[a, b]}= \sum_{j=0}^k \|h^{(j)}\|_{C[a, b]} \ {\rm for} \  h\in C^k[a, b].$$
Clearly,  $C[a, b]$  and $C^k[a, b]$ are Banach algebras under function multiplication.
Moreover  
  \begin{eqnarray}\label{Cab.eq}
\|h_1h_2\|_{C^1[a,b]} & = & \|(h_1h_2)'\|_{C[a, b]}+ \|h_1h_2\|_{C[a, b]}\nonumber\\
& \le &
\|h_1'\|_{C[a, b]} \|h_2\|_{C[a, b]}+ \|h_1\|_{C[a, b]} \|h_2'\|_{C[a, b]}
 + \|h_1\|_{C[a, b]}\|h_2\|_{C[a, b]}\nonumber\\
& \le &  \|h_1\|_{C^1[a,b]} \|h_2\|_{C[a,b]}+\|h_1\|_{C[a,b]}\|h_2\|_{C^1[a,b]} \ {\rm for \ all} \ h_1, h_2\in C^1[a, b],
\end{eqnarray}
where the second inequality follows from the Leibniz rule.  Therefore we have

\begin{thm}\label{C1ab.thm}
 $C^1[a, b]$ is a differential subalgebra of order one in $C[a, b]$.
 \end{thm}
  Due to the above illustrative example of differential subalgebras of order one,
the  norm $\|\cdot\|_{\mathcal A}$  satisfying \eqref{differentialnorm.def} 
is also used to describe smoothness in abstract Banach algebra  \cite{blackadarcuntz91}.

Let
${\mathcal W}^1$  be  the Banach algebra of all periodic functions such that
both $f$ and its derivative $f'$ belong to  the Wiener algebra  ${\mathcal W}$, and define the norm on ${\mathcal W}^1$ by
\begin{equation}\label{differentialwiener.def}
\|f\|_{{\mathcal W}^1}  =  \|f\|_{\mathcal W}+\|f'\|_{\mathcal W}
  =  \sum_{n\in \ZZ} (|n|+1) |\hat f(n)|
  \end{equation}
  for $f(x)=\sum_{n\in \ZZ} \hat f(n) e^{inx}\in {\mathcal W}^1$.
Following the argument used in the proof of Theorem \ref{C1ab.thm}, 
we have
\begin{thm} \label{wienerdiff.theorem}
    ${\mathcal W}^1$ is a differential subalgebra of order one in ${\mathcal W}$.
    \end{thm}

Recall from  the classical Wiener's lemma that  ${\mathcal W}$ is an inverse-closed subalgebra of
${\mathcal C}$, the algebra of all periodic continuous functions under multiplication.
This leads to the following natural question:

\begin{ques}\label{question1}
 Is ${\mathcal W}^1$  a differential subalgebra
of ${\mathcal C}$?
\end{ques}

Let $\ell^p, 1\le p\le \infty$, be
 the space of all $p$-summable  sequences on $\ZZ$ with norm denoted by $\|\cdot\|_p$.
  To answer the above question,
we consider  Banach algebras ${\mathcal C}$, ${\mathcal W}$ and ${\mathcal W}^1$  in the ``frequency domain".
Let ${\mathcal B}(\ell^p)$ be the algebra of all bounded linear operators on  $\ell^p, 1\le p\le \infty$,
 and let   
\begin{equation}
\label{tildew.def}
\tilde {\mathcal W}=\Big\{A:=(a(i-j))_{i,j\in \ZZ},\  \| A\|_{\tilde W}=\sum_{k\in \ZZ} |a(k)|<\infty\Big\}\end{equation}
and
\begin{equation}
\label{tildew1.def}
{\tilde {\mathcal W}}^1=\Big\{A:=(a(i-j))_{i,j\in \ZZ}, \ \| A\|_{{\tilde W}^1}=\sum_{k\in \ZZ} |k| |a(k)|<\infty\Big\}\end{equation}
be  Banach algebras of  Laurent matrices
with symbols in ${\mathcal W}$ and ${\mathcal W}^1$ respectively.   Then the classical Wiener's lemma can be reformulated
as that $\tilde {\mathcal W}$ is an inverse-closed subalgebra of ${\mathcal B}(\ell^2)$,
and  an equivalent statement of Theorem \ref{wienerdiff.theorem}
is that ${\tilde {\mathcal W}}^1$ is a differential subalgebra of order one in $\tilde {\mathcal W}$.
Due to the above equivalence,  Question
\ref{question1} in the ``frequency domain" becomes whether  ${\mathcal W}^1$ is a differential subalgebra of order $\theta\in (0, 1]$ in  ${\mathcal C}$.
In  \cite{suncasp05}, the first example of  differential subalgebra of infinite matrices with order $\theta\in (0, 1)$
was discovered.

\begin{thm}\label{W1.thm}
${\mathcal W}^1$ is a differential subalgebra of ${\mathcal C}$ with order $2/3$.
\end{thm}

To consider differential subalgebras  of infinite matrices in the noncommutative setting, we introduce three noncommutative  Banach algebras of
infinite matrices with certain off-diagonal decay.
Given $1\le p\le \infty$ and $\alpha\ge 0$,
 we define
the Gr\"ochenig-Schur  family of infinite matrices
by
 \begin{equation}\label{GS.def}
{\mathcal A}_{p,\alpha}=\Big\{ A=(a(i,j))_{i,j \in \Z}, \  \|A\|_{{\mathcal A}_{p,\alpha}}<\infty\Big\}
\end{equation}
\cite{gltams06, jaffard90, moteesun, schur11,  suntams07, suncasp05},
the
Baskakov-Gohberg-Sj\"ostrand family of infinite matrices by
\begin{equation}\label{BGS.def}
{\mathcal C}_{p,\alpha}=\Big\{ A= (a(i,j))_{i,j \in \Z}, \  \|A\|_{{\mathcal C}_{p,\alpha}}<\infty\Big\}
\end{equation}
\cite{baskakov90, gkwieot89, gltams06, sjostrand94,suntams07}, and
the Beurling family of infinite matrices
\begin{equation}\label{Beurling.def}
{\mathcal B}_{p,\alpha}=\Big\{ A= (a(i,j))_{i,j \in \Z}, \  \|B\|_{{\mathcal A}_{p,\alpha}}<\infty\Big\}
\end{equation}
\cite{beurling49,  shinsun19, sunca11},
 where     $u_\alpha(i, j)=(1+|i-j|)^\alpha, \alpha\ge 0$, are  polynomial weights on $\Z^2$,
 \begin{equation}\label{GSnorm.def}
 \|A\|_{{\mathcal A}_{p,\alpha}}
 = \max \Big\{ \sup_{i \in \Z} \big\|\big(a(i,j) u_\alpha(i, j)\big)_{j\in \Z}\big\|_p, \ \ \sup _{j \in \Z}
 \big\|\big(a(i,j) u_\alpha(i, j)\big)_{i\in \Z}\big\|_p
 \Big\},
\end{equation}
\begin{equation}\label{BKSnorm.def}
\|A\|_{{\mathcal C}_{p,\alpha}} = \Big\| \Big(\sup_{i-j=k} |a(i,j)| u_\alpha(i, j)\Big)_{k\in \Z} \Big\|_p,
\end{equation}
and
\begin{equation}\label{Beurlingnorm.def}
\|A\|_{{\mathcal B}_{p,\alpha}} = \Big\| \Big(\sup_{|i-j|\ge |k| } |a(i,j)| u_\alpha(i, j)\Big)_{k\in \Z} \Big\|_p.
\end{equation}
Clearly, we have
\begin{equation}\label{properinclusion}
{\mathcal B}_{p,\alpha} \subset {\mathcal C}_{p,\alpha} \subset
{\mathcal A}_{p,\alpha}  \ \ {\rm for \ all}\ 1\le p\le \infty \ {\rm and} \ \alpha\ge 0.
\end{equation}
The above inclusion  is proper for $1\le p<\infty$, while
the  above three  families  
of infinite matrices coincide  for $p=\infty$,
\begin{equation}\label{properinclusioninfinite}
{\mathcal B}_{\infty,\alpha}={\mathcal C}_{\infty,\alpha}=
{\mathcal A}_{\infty,\alpha}  \ \ {\rm for \ all} \ \alpha\ge 0,
\end{equation}
which is also known as the Jaffard family of infinite matrices  \cite{jaffard90},
 \begin{equation}\label{Jaffard.def}
{\mathcal J}_{\alpha}=\Big\{ A= (a(i,j))_{i,j \in \Z}, \  \|A\|_{{\mathcal J}_{\alpha}}=\sup_{i,j\in \Z}|a(i,j)| u_\alpha(i-j)<\infty\Big\}.
\end{equation}

Observe that
$\|A\|_{{\mathcal A}_{p,\alpha}}=\|A\|_{{\mathcal C}_{p,\alpha}}$
for a Laurent matrix  $A=(a(i-j))_{i,j\in \Z}$. Then
Banach algebras $\tilde {\mathcal W}$ and ${\tilde {\mathcal W}}^1$
in \eqref{tildew.def} and \eqref{tildew1.def}
are the commutative subalgebra of the  Gr\"ochenig-Schur  algebra ${\mathcal A}_{1, \alpha}$
and the Baskakov-Gohberg-Sj\"ostrand algebra ${\mathcal C}_{1, \alpha}$ for $\alpha=0, 1$ respectively,
\begin{equation}\label{wa.re}
\tilde {\mathcal W}= {\mathcal A}_{1, 0}\cap {\mathcal L}={\mathcal C}_{1, 0}\cap {\mathcal L}
\end{equation}
and
\begin{equation}\label{wa1.re}
{\tilde {\mathcal W}}^1= {\mathcal A}_{1, 1}\cap {\mathcal L}={\mathcal C}_{1, 1}\cap {\mathcal L},
\end{equation}
where ${\mathcal L}$ is the set of all Laurent matrices  $A=(a(i-j))_{i,j\in \Z}$.
The sets ${\mathcal A}_{p, \alpha}, {\mathcal C}_{p,\alpha}, {\mathcal B}_{p,\alpha}$
with $p=1$ and $\alpha=0$ are  noncommutative Banach algebras under matrix multiplication,
the Baskakov-Gohberg-Sj\"ostrand algebra  ${\mathcal C}_{1,0}$ and the Beurling algebra ${\mathcal B}_{1, 0}$ are inverse-closed subalgebras of ${\mathcal B}(\ell^2)$ \cite{baskakov90, bochnerphillips42, gkwieot89,  sjostrand94, sunca11}, however
the Schur algebra ${\mathcal A}_{1,0}$ is not inverse-closed in ${\mathcal B}(\ell^2)$  \cite{tessera10}.
We remark that the inverse-closedness of the Baskakov-Gohberg-Sj\"ostrand algebra  ${\mathcal C}_{1,0}$
in ${\mathcal B}(\ell^2)$ can be understood as a noncommutative extension of the classical Wiener's lemma for the
commutative subalgebra $\tilde {\mathcal W}$ of Laurent matrices in ${\mathcal B}(\ell^2)$.

For $1\le p\le \infty$ and $\alpha>1-1/p$, one may verify that
the Gr\"ochenig-Schur family ${\mathcal A}_{p,\alpha}$,
the Baskakov-Gohberg-Sj\"ostrand family  ${\mathcal C}_{p,\alpha}$ and the Beurling family ${\mathcal B}_{p, \alpha}$
of infinite matrices form Banach algebras  under  matrix multiplication
and they are inverse-closed subalgebras of ${\mathcal B}(\ell^2)$ \cite{gltams06, jaffard90, sunca11, suntams07, suncasp05}.
In \cite{sunca11, suntams07, suncasp05},  their differentiability in ${\mathcal B}(\ell^2)$ is established.

\begin{thm} \label{sundiff.thm}
Let $1\le p\le \infty$ and $\alpha>1-1/p$. Then
 ${\mathcal A}_{p,\alpha}$,
  ${\mathcal C}_{p,\alpha}$ and  ${\mathcal B}_{p, \alpha}$ are differential subalgebras of  order
  $\theta_0= (\alpha+1/p-1)/(\alpha+1/p-1/2)\in (0, 1)$ in ${\mathcal B}(\ell^2)$.
\end{thm}

\begin{proof} The following argument about differential subalgebra property for the Gr\"ochenig-Schur algebra ${\mathcal A}_{p, \alpha}, 1<p<\infty$, is adapted from \cite{suncasp05}.  The reader may refer to \cite{sunca11, suntams07, suncasp05} for the detailed proof to the
differential subalgebra property for the Baskakov-Gohberg-Sj\"ostrand algebra
  ${\mathcal C}_{p,\alpha}$ and  the Beurling algebra ${\mathcal B}_{p, \alpha}$.
 Take $A=(a(i,j))_{i,j\in \Z}$ and $B=(b(i,j))_{i,j\in \Z}\in {\mathcal A}_{p, \alpha}$, and write
 $C=AB=(c(i,j))_{i,j\in \Z}$. Then
 \begin{eqnarray}\label{sundiff.thm.pf.eq1}
 \|C\|_{{\mathcal A}_{p, \alpha}} & = &
 \max \Big\{ \sup_{i \in \Z} \big\|\big(c(i,j) u_\alpha(i, j)\big)_{j\in \Z}\big\|_p, \ \ \sup _{j \in \Z}
 \big\|\big(c(i, j) u_\alpha(i, j)\big)_{i\in \Z}\big\|_p\Big\}\nonumber\\
 &\le & 2^\alpha \max \Big\{ \sup_{i \in \Z} \Big\|\Big(\sum_{k\in \Z} |a(i,k)| |b(k,j)| \big(u_\alpha(i, k)+u_\alpha(k, j)\big) \Big)_{j\in \Z}\Big\|_p, \nonumber\\
  & & \qquad \qquad \ \  \sup_{j \in \Z} \Big\|\Big(\sum_{k\in \Z} |a(i,k)| |b(k,j)| \big(u_\alpha(i, k)+u_\alpha(k, j)\big) \Big)_{i\in \Z}\Big\|_p
 \Big\}\nonumber\\
 & \le & 2^\alpha \|A\|_{{\mathcal A}_{p, \alpha}} \|B\|_{{\mathcal A}_{1, 0}}+ 2^\alpha \|A\|_{{\mathcal A}_{1, 0}} \|B\|_{{\mathcal A}_{p, \alpha}},
 \end{eqnarray}
 where the first inequality follows from the inequality
 $$u_\alpha(i,j)\le  2^\alpha \big(u_\alpha(i, k)+ u_\alpha(k, j)), \ i, j, k\in \ZZ.$$

Let $1/p'=1-1/p$,  and define
   \begin {equation}
   \label{tau0.def}
   \tau_0= \left\lfloor \left(\Bigg(\frac{\alpha p'+1}{\alpha p'-1}\Bigg)^{1/p'}\frac{\|A\|_{{\mathcal A}_{p, \alpha}}}{\|A\|_{{\mathcal B}(\ell^2)}}\right)^{1/(\alpha+1/2-1/p')}\right\rfloor,\end{equation}
   where $\lfloor t\rfloor$ denotes the integer part of a real number $t$.
      Then for $i\in \Z$, we have
 \begin{eqnarray}\label{para1.eq}
 \sum_{j\in \Z} |a(i,j)| & = & \Big(\sum_{|j-i|\le \tau_0}+\sum_{|j-i|>\tau_0}\Big) |a(i,j)|\nonumber\\
  & \le &  \Big(\sum_{ |j-i|\le \tau_0} |a(i,j)|^2\Big)^{1/2} \Big(\sum_{|j-i|\le \tau_0} 1\Big)^{1/2}\nonumber\\
 &  & + \Big(\sum_{|j-i|\ge \tau_0+1} |a(i,j)|^p (u_\alpha(i, j))^p\Big)^{1/p} \Big(\sum_{|j-i|\ge \tau_0+1} (u_\alpha(i, j))^{-p'} \Big)^{1/p'}\nonumber\\
 & \le &  \|A\|_{{\mathcal B}(\ell^2)} (2\tau_0+1)^{1/2}+  2^{1/p'} (\alpha p'-1)^{-1/p'}  \|A\|_{{\mathcal A}_{p, \alpha}}
 (\tau_0+1)^{-\alpha+1/p'}\nonumber\\
 & \le &  D
 \|A\|_{{\mathcal A}_{p, \alpha}}^{1-\theta_0}  \|A\|_{{\mathcal B}(\ell^2)}^{\theta_0},
   \end{eqnarray}
   where $D$ is an absolute constant depending on $p, \alpha$ only, and
   the last inequality follows from  \eqref{tau0.def} and the following estimate
     $$   \|A\|_{{\mathcal B}(\ell^2)}\le \|A\|_{{\mathcal A}_{1, 0}}\le \Big(\sum_{k\in \Z} (|k|+1)^{-\alpha p'}\Big)^{1/p'}
   \|A\|_{{\mathcal A}_{p, \alpha}}\le \Big(\frac{\alpha p'+1}{\alpha p'-1}\Big)^{1/p'}\|A\|_{{\mathcal A}_{p, \alpha}}.$$
   Similarly we can prove that 
   \begin{equation}
   \label{para2.eq}
  \sup_{j\in \Z} \sum_{i\in \Z} |a(i,j)| \le D
 \|A\|_{{\mathcal A}_{p, \alpha}}^{1-\theta_0}  \|A\|_{{\mathcal B}(\ell^2)}^{\theta_0}.\end{equation}
 Combining \eqref{para1.eq} and \eqref{para2.eq} leads to
   \begin{equation}
   \label{para3.eq}
  \|A\|_{{\mathcal A}_{1, 0}} \le D
 \|A\|_{{\mathcal A}_{p, \alpha}}^{1-\theta_0}  \|A\|_{{\mathcal B}(\ell^2)}^{\theta_0}.\end{equation}
Replacing the matrix $A$ in \eqref{para3.eq} by the matrix $B$ gives
   \begin{equation}
   \label{para4.eq}
  \|B\|_{{\mathcal A}_{1, 0}} \le D
 \|B\|_{{\mathcal A}_{p, \alpha}}^{1-\theta_0}  \|B\|_{{\mathcal B}(\ell^2)}^{\theta_0}.\end{equation}
Therefore it follows from \eqref{sundiff.thm.pf.eq1}, \eqref{para3.eq} and \eqref{para4.eq} that
\begin{equation}
\|C\|_{{\mathcal A}_{p, \alpha}}\le  2^\alpha  D
 \|A\|_{{\mathcal A}_{p, \alpha}}\|B\|_{{\mathcal A}_{p, \alpha}}^{1-\theta_0}  \|B\|_{{\mathcal B}(\ell^2)}^{\theta_0}
+2^\alpha  D
 \|B\|_{{\mathcal A}_{p, \alpha}}\|A\|_{{\mathcal A}_{p, \alpha}}^{1-\theta_0}  \|A\|_{{\mathcal B}(\ell^2)}^{\theta_0},
\end{equation}
which proves the differential subalgebra property for Banach algebras ${\mathcal A}_{p, \alpha}$ with $1<p<\infty$ and $\alpha>1-1/p$.
\end{proof}

The argument used in the proof of Theorem \ref{sundiff.thm} involves a triplet of  Banach algebras
${\mathcal A}_{p, \alpha}$, ${\mathcal A}_{1, 0}$ and ${\mathcal B}^2$ satisfying \eqref{sundiff.thm.pf.eq1} and
\eqref{para3.eq}.
In the following theorem, we extend the above observation to
 general Banach algebra triplets  $({\mathcal A}, {\mathcal M}, {\mathcal B})$.

\begin{thm}\label{triple1.thm}
 Let ${\mathcal A}, {\mathcal M}$  and
${\mathcal B}$ be Banach algebras such that ${\mathcal A}$ is a Banach subalgebra of ${\mathcal M}$
and ${\mathcal M}$ is a  Banach subalgebra of ${\mathcal B}$.
If there exist positive exponents $\theta_0, \theta_1\in (0, 1]$ and  absolute constants $D_0, D_1$ such that
\begin{equation}\label{triple1.thm.eq1}
\|AB\|_{\mathcal A}\le D_0\|A\|_{\mathcal A} \|B\|_{\mathcal A} \Big (\Big(\frac{\|A\|_{\mathcal M}}{\|A\|_{\mathcal A}}\Big)^{\theta_0} +
\Big(\frac{\|B\|_{\mathcal M}}{\|B\|_{\mathcal A}}\Big)^{\theta_0}
\Big)
\quad {\rm for \ all} \  A, B \in {\mathcal A},
\end{equation}
 and
\begin{equation}\label{triple1.thm.eq2}
\|A\|_{\mathcal M}\le D_1 \|A\|_{\mathcal A}^{1-\theta_1} \|A\|_{\mathcal B}^{\theta_1} \ \ {\rm for \ all} \   \ A\in {\mathcal A},
\end{equation}
then
${\mathcal A}$ is a differential subalgebra of order  $\theta_0\theta_1$ in ${\mathcal B}$.
\end{thm}

\begin{proof}  For any $A, B\in {\mathcal A}$, we obtain from \eqref{triple1.thm.eq1} and
\eqref{triple1.thm.eq2} that
\begin{eqnarray*}
\|AB\|_{\mathcal A} & \le &  D_0\|A\|_{\mathcal A} \|B\|_{\mathcal A} \Bigg (\Big(\frac{D_1 \|A\|_{\mathcal A}^{1-\theta_1} \|A\|_{\mathcal B}^{\theta_1}}{\|A\|_{\mathcal A}}\Big)^{\theta_0} +
\Big(\frac{D_1 \|B\|_{\mathcal A}^{1-\theta_1} \|B\|_{\mathcal B}^{\theta_1}}{\|B\|_{\mathcal A}}\Big)^{\theta_0}
\Bigg)\\
  & \le &  D_0  D_1^{\theta_0}\|A\|_{\mathcal A} \|B\|_{\mathcal A} \Big (\Big(\frac{\|A\|_{\mathcal B}}{\|A\|_{\mathcal A}}\Big)^{\theta_0\theta_1} +
\Big(\frac{\|B\|_{\mathcal B}}{\|B\|_{\mathcal A}}\Big)^{\theta_0\theta_1}
\Big),
\end{eqnarray*}
which completes the proof.
\end{proof}

Following the argument used in \eqref{Cab.eq}, we can show that $C^2[a, b]$ is a differential subalgebra of $C^1[a, b]$.
For any  distinct $x, y\in [a, b]$ and $f\in C^2[a, b]$,  observe that
$$
|f'(x)|= \frac{|f(y)-f(x)-f''(\xi) (y-x)^2/2 |}{|y-x|} \le 2\|f\|_{C[a, b]} |y-x|^{-1}+ \frac{1}{2} \|f^{\prime\prime}\|_{C[a, b]}
|y-x|
$$
for some $\xi\in [a, b]$, which implies that
\begin{equation}
\|f'\|_{C[a, b]}\le \max\big (4 \|f\|_{C[a, b]}^{1/2} \|f^{\prime\prime} \|_{C[a, b]}^{1/2},  8 (b-a)^{-1} \|f\|_{C[a, b]}\big).
\end{equation}
Therefore there exists  a positive constant $D$ such that
\begin{equation}
\|f\|_{C^1[a,b]}\le D \|f\|_{C^2[a, b]}^{1/2} \|f\|_{C[a, b]}^{1/2} \ \ {\rm for \ all} \   \ f\in  C^2[a, b].
\end{equation}
As an application of Theorem  \ref{triple1.thm}, we conclude that
$C^2[a, b]$ is a differential subalgebra of  order $1/2$ in $C[a, b]$.

\smallskip

We finish the section with the proof of Theorem \ref{W1.thm}.

\begin{proof}[Proof of Theorem \ref{W1.thm}]
The conclusion follows  from \eqref{wa1.re} and Theorem \ref{sundiff.thm} with $p=1$ and $\alpha=1$.
\end{proof}

\section{Generalized differential subalgebras}\label{generalizedDS.section}

By \eqref{differentialnorm.def}, a  differential subalgebra ${\mathcal A}$ satisfies the Brandenburg's  requirement:
\begin{equation}\label{bt.req}
\|A^2\|_{\mathcal A}\le 2D_0 \|A\|_{\mathcal A}^{2-\theta} \|A\|_{\mathcal B}^{\theta}, \ A\in {\mathcal A}.
\end{equation}
To consider  the norm-controlled
inversion of a Banach subalgebra ${\mathcal A}$ of ${\mathcal B}$,
the above requirement \eqref{bt.req} could be relaxed to the existence of an integer $m\ge 2$ such that
the $m$-th power of elements in ${\mathcal A}$ satisfies
\begin{equation}\label{weakpower}
\|A^m\|_{\mathcal A} \le D \|A\|_{\mathcal A}^{m-\theta} \|A\|_{\mathcal B}^{\theta}, \ \  A\in {\mathcal A},
\end{equation}
where  $\theta\in (0, m-1]$ and $D=D({\mathcal A}, {\mathcal B}, m, \theta)$ is an absolute positive constant, see   Theorem \ref{main-thm1} in the next section.
For $h\in C^1[a, b]$ and $m\ge 2$, we have
  \begin{equation*}
\|h^m \|_{C^1[a,b]}  =  m \|h^{m-1} h'\|_{C[a, b]}+ \|h^m\|_{C[a, b]}
\le m \|h\|_{C^1[a, b]} \|h\|_{C[a, b]}^{m-1},
\end{equation*}
and hence  the differential subalgebra $C^1[a, b]$ of $C[a, b]$ satisfies
\eqref{weakpower} with $\theta=m-1$.
In this section, we introduce some sufficient conditions 
so that
\eqref{weakpower} holds for some integer $m\ge 2$.

\begin{thm}\label{triplenew.thm}   Let ${\mathcal A}, {\mathcal M}$  and
${\mathcal B}$ be Banach algebras such that ${\mathcal A}$ is a Banach subalgebra of ${\mathcal M}$
and ${\mathcal M}$ is a Banach subalgebra of ${\mathcal B}$.
If there exist an integer $k\ge 2$, positive exponents $\theta_0, \theta_1$, and  absolute constants $E_0, E_1$
such that
\begin{equation}\label{triplenew.eq1}
\|A_1A_2\cdots A_k\|_{\mathcal A}  \le   E_0
\Big(\prod_{i=1}^k\|A_i\|_{{\mathcal A}} \Big) \sum_{j=1}^k \Big(\frac{\|A_i\|_{\mathcal M}}{\|A_i\|_{\mathcal A}}\Big)^{\theta_0}, \ \  A_1, \ldots, A_k \in {\mathcal A}
\end{equation}
and
\begin{equation}\label{triplenew.eq2}
\|A^2\|_{\mathcal M}\le E_1  \|A\|_{\mathcal A}^{2-\theta_1} \|A\|_{\mathcal B}^{\theta_1},\  A\in {\mathcal A},
\end{equation}
then \eqref{weakpower} holds for $m=2k$ and  $\theta=\theta_0\theta_1$.
\end{thm}

\begin{proof}   By \eqref{banachalgebra.def},
\eqref{triplenew.eq1} and \eqref{triplenew.eq2}, we have
\begin{equation}
\|A^{2k}\|_{\mathcal A}  \le     k E_0
\|A^2\|_{{\mathcal A}}^{k-\theta_0} \|A^2\|_{\mathcal M}^{\theta_0}
\le k E_0 E_1^{\theta_0}  K^{k-\theta_0} \|A\|_{{\mathcal A}}^{2k-\theta_0\theta_1} \|A\|_{\mathcal B}^{\theta_0\theta_1}, \ \  A\in {\mathcal A},
\end{equation}
which completes the proof. 
\end{proof}

For a Banach algebra triplet $({\mathcal A},  {\mathcal M}, {\mathcal B})$  in Theorem \ref{triple1.thm}, we obtain from
\eqref{triple1.thm.eq1} and \eqref{triple1.thm.eq2} that
\begin{eqnarray}\label{triple1.eq3}
\|A_1A_2\cdots A_k\|_{\mathcal A} & \le &  D_0
\|A_1\|_{{\mathcal A}} \|A_2\cdots A_k\|_{{\mathcal A}}
\Bigg (\Big(\frac{\|A_1\|_{\mathcal M}}{\|A_1\|_{\mathcal A}}\Big)^{\theta_0} +
\Big(\frac{\|A_2\cdots A_k\|_{\mathcal M}}{\|A_2\cdots A_k\|_{\mathcal A}}\Big)^{\theta_0}
\Bigg)\nonumber\\
& \le & \tilde D_0
\Big(\prod_{i=1}^k\|A_i\|_{{\mathcal A}} \Big)
\sum_{j=1}^k \Big(\frac{\|A_j\|_{\mathcal M}}{\|A_j\|_{\mathcal A}}\Big)^{\theta_0}, \ \ A_1, \ldots, A_k\in {\mathcal A},
\end{eqnarray}
and
\begin{equation}\label{triple1.eq4}
\|A^2\|_{\mathcal M}\le \tilde K \|A\|_{\mathcal M}^2\le
D_1^2 \tilde K \| A\|_{\mathcal A}^{2-2\theta_1}\|A\|_{\mathcal B}^{2\theta_1}, \ \ A\in {\mathcal A},
\end{equation}
where $\tilde D_0$ is an absolute constant and $\tilde K$ is the constant $K$ in
\eqref{banachalgebra.def} for the Banach algebra ${\mathcal M}$.
Therefore the assumptions  \eqref{triplenew.eq1} and \eqref{triplenew.eq2}
in Theorem \ref{triplenew.thm} are satisfied
for the Banach algebra triplet $({\mathcal A},  {\mathcal M}, {\mathcal B})$  in Theorem \ref{triple1.thm}.

For a differential subalgebra  ${\mathcal A}$ of order $\theta_0$ in ${\mathcal B}$,
we observe  that the requirements
\eqref{triplenew.eq1} and \eqref{triplenew.eq2} with  ${\mathcal M}={\mathcal B}$,  $k=2$ and $\theta_1=2$ are met,
 and hence
 \eqref{weakpower} holds for $m=4$ and  $\theta=2\theta_0$.
Recall that ${\mathcal B}$ is a trivial differential subalgebra of  ${\mathcal B}$.
In the following corollary,  we can extend the above conclusion to arbitrary  differential subalgebras ${\mathcal M}$ of ${\mathcal B}$.

\begin{cor}
Let ${\mathcal A}, {\mathcal M}$  and
${\mathcal B}$ be Banach algebras such that ${\mathcal A}$ is a differential subalgebra of order $\theta_0$ in  ${\mathcal M}$
and ${\mathcal M}$ is a differential subalgebra of  order $\theta_1$ in ${\mathcal B}$.
Then \eqref{weakpower} holds for $m=4$ and  $\theta=\theta_0\theta_1$.
\end{cor}

Following the argument used in the proof of Theorem \ref{triplenew.thm}, we can show that
\eqref{weakpower} holds for $m=4$ if the requirement \eqref{triplenew.eq1} with $k=3$ is replaced by the following strong version
\begin{equation}\label{triplenew.eq3}
\|ABC\|_{\mathcal A}  \le   E_0
\|A\|_{{\mathcal A}}  \|C\|_{{\mathcal A}}  \|B\|_{{\mathcal A}}^{1-\theta_0}
\|B\|_{\mathcal M}^{\theta_0}, \ \  A, B, C \in {\mathcal A}.
\end{equation}

\begin{thm}\label{triplenew.thm2}
Let ${\mathcal A}, {\mathcal M}$  and
${\mathcal B}$ be Banach algebras such that ${\mathcal A}$ is a Banach subalgebra of ${\mathcal M}$
and ${\mathcal M}$ is a Banach subalgebra of ${\mathcal B}$.
If there exist positive exponents $\theta_0, \theta_1\in (0, 1]$ and  absolute constants $E_0, E_1$
such that \eqref{triplenew.eq2} and \eqref{triplenew.eq3} hold,
then \eqref{weakpower} holds for $m=4$ and  $\theta=\theta_0\theta_1$.
\end{thm}

Let $L^p:=L^p(\R), 1\le p\le \infty$, be the space of all $p$-integrable functions on $\R$  with standard norm $\|\cdot\|_p$,
and ${\mathcal B}(L^p)$ be the algebra of bounded linear operators on
$L^p$ with the norm $\|\cdot \|_{{\mathcal B}(L^p)}$.
For $1\le p\le \infty, \alpha\ge 0$ and $\gamma\in [0, 1)$, we define the norm
of a kernel  $K$ on $\R\times \R$ by
\begin{equation}\label{Ex3-def-norm}
\|K\|_{{\mathcal W}_{p,\alpha}^\gamma}=
\left\{
\begin{array}{ll}
 \max\Big(\sup_{x\in \R} \big\|K(x,\cdot)u_\alpha(x,\cdot)\big\|_p,\
\sup_{y\in \R} \big\|K(\cdot,y)u_\alpha(\cdot,y)\big\|_p\Big)  & {\rm  if }\ \gamma =0
\\
\|K\|_{{\mathcal W}_{p,\alpha}^0}+\sup_{0<\delta\le 1} \delta^{-\gamma}
\|\omega_\delta(K)\|_{{\mathcal W}_{p,\alpha}^0} & {\rm if } \ 0 < \gamma <1,
\end{array}
\right.
\end{equation}
where  the modulus of continuity
of  the  kernel  $K$ is defined by
\begin{equation}\label{Ex3-def-mod}
\omega_\delta(K)(x,y):=\sup_{|x^\prime|\le \delta, |y^\prime|\le \delta}
|K(x+x^\prime, y+y^\prime)-K(x,y)|, \ x, y\in \RR,
\end{equation}
and  $u_\alpha(x, y)= (1+|x-y|)^\alpha,  x, y\in \R$ are polynomial weights on $\R\times \R$.
Consider the set  ${\mathcal W}_{p,\alpha}^\gamma$ of integral operators
\begin{equation}\label{Ex3-def-int-oper}
Tf(x)=\int_{{\R}} K_T(x,y) f(y) dy, \quad f \in L^p,
\end{equation}
whose integral kernels $K_T$ satisfy $\|K_T\|_{{\mathcal W}^\gamma_{p, \alpha}}<\infty$, and define
$$
\|T\|_{{\mathcal W}_{p,\alpha}^\gamma}:=
\|K_T \|_{{\mathcal W}_{p,\alpha}^\gamma}, \ T\in {\mathcal W}_{p,\alpha}^\gamma.
$$
Integral operators in ${\mathcal W}_{p, \alpha}^\gamma$ have their kernels being H\"older continuous of order $\gamma$
 and having off-diagonal polynomial decay of order $\alpha$.
For $1\le p\le \infty$ and $\alpha>1-1/p$, one may verify that
 ${\mathcal W}_{p, \alpha}^\gamma, 0\le \gamma<1$, are Banach subalgebras of
  ${\mathcal B}(L^2)$ under operator composition.
  The Banach algebras  ${\mathcal W}_{p, \alpha}^\gamma, 0<\gamma<1$, of  integral operators
  may not form a differential subalgebra of ${\mathcal B}(L^2)$, however  the triple
$({\mathcal W}_{p, \alpha}^\gamma,  {\mathcal W}_{p, \alpha}^0, {\mathcal B}(L^2))$ is proved in
 \cite{sunacha08} to satisfy the following
\begin{equation}\label{Ex3-norm-comp}
\|T_0\|_{\mathcal B} \le  D \|T_0\|_{{\mathcal W}_{p,\alpha}^0} \le
D\|T_0\|_{{\mathcal W}_{p,\alpha}^\gamma},
\end{equation}
\begin{equation}\label{Ex3-norm-2-theta}
\|T_0^2 \|_{{\mathcal W}_{p,\alpha}^0} \le D
\|T_0 \|_{{\mathcal W}_{p,\alpha}^\gamma}^{1+\theta}  \|T_0 \|_{{\mathcal B}(L^2)}^{1-\theta}
\end{equation}
and
\begin{equation}\label{Ex3-norm-3-product}
\|T_1 T_2 T_3  \|_{{\mathcal W}_{p,\alpha}^\gamma} \le D
\|T_1 \|_{{\mathcal W}_{p,\alpha}^\gamma} \| T_2 \|_{{\mathcal W}_{p,\alpha}^0}
\| T_3 \|_{{\mathcal W}_{p,\alpha}^\gamma}
\end{equation}
holds for all $T_i\in {\mathcal W}_{p, \alpha}^\gamma, 0\le i\le 3$, where $D$ is an absolute constant and
$$\theta= \frac{\alpha+\gamma+1/p}{(1+\gamma)(\alpha+1/p)}.$$
Then  the requirements  \eqref{triplenew.eq2} and \eqref{triplenew.eq3} in Theorem \ref{triplenew.thm2}
are met for  the triplet $({\mathcal W}_{p, \alpha}^\gamma,  {\mathcal W}_{p, \alpha}^0, {\mathcal B}(L^2))$,
and hence
the Banach space pair $({\mathcal W}_{p, \alpha}^\gamma,   {\mathcal B}(L^2))$ satisfies the Brandenburg's  condition \eqref{weakpower} with $m=4$
\cite{fangshinsun13, sunacha08}.

\section{Brandenburg  trick and norm-controlled inversion}\label{normcontrolledinversion.section}

Let
$\mathcal{A}$  and $\mathcal{B}$ are $*$-algebras with common identity and involution, and  let
 $\mathcal{B}$ be  symmetric. In this section, we show that
 ${\mathcal A}$ has norm-controlled inversion in ${\mathcal B}$ if it meets
 the Brandenburg requirement \eqref{weakpower}.

\begin{thm}\label{main-thm1}
Let  ${\mathcal B}$  be a   symmetric $*$-algebra with its norm
$\|\cdot\|_{\mathcal B}$
being normalized in the sense that \eqref{banachalgebra.def} holds with $K=1$,
 \begin{equation}\label{banachalgebra.defnew2}
\|\tilde A\tilde B\|_{\mathcal B}\le \|\tilde A\|_{\mathcal B}\|\tilde B\|_{\mathcal B},\   \tilde  A, \tilde B\in {\mathcal B},
\end{equation}
and
$\mathcal{A}$   be  a $*$-algebra with its
norm
$\|\cdot\|_{\mathcal A}$
being normalized too,
\begin{equation}\label{banachalgebra.defnew1}
\|AB\|_{\mathcal A}\le \|A\|_{\mathcal A}\|B\|_{\mathcal A}, \  A, B\in {\mathcal A}.
\end{equation}
If  ${\mathcal A}$ is a $*$-subalgebra of ${\mathcal B}$  with  common  identity  $I$ and involution  $*$,  and
the  pair $({\mathcal A}, {\mathcal B})$ satisfies
the Brandenburg requirement \eqref{weakpower}, then
${\mathcal A}$ has norm-controlled inversion in ${\mathcal B}$. Moreover,
for any  $A\in{\mathcal A}$  being  invertible in ${\mathcal B}$ we have
\begin{eqnarray}\label{norm-control1}
\|A^{-1}\|_{\mathcal A}
& \le &
\|A^* A\|_{\mathcal B}^{-1}  \|A^*\|_{\mathcal A}\nonumber\\
& & \times
\left\{\begin{array}{ll}
   \big(2t_0+(1-  2^{\log_m (1-\theta/m)})^{-1} (\ln a)^{-1}\big)  a
 \exp\Big(\frac{\ln m-\ln (m-\theta)} {\ln (m-\theta)} t_0 \ln a\Big)
  & {\rm if} \ \theta<m-1\\
a^2 (\ln a)^{-1}  (Db)^{m-1} \Gamma\Big(\frac{(m-1) \ln (Db)}{\ln m \ln a}+1\Big)
&   {\rm if} \ \theta=m-1,
\end{array}\right.
\end{eqnarray}
where $\Gamma(s)=\int_0^\infty t^{s-1} e^{-t} dt$ is the Gamma function, $m\ge 2$ and $0<\theta\le m-1$ are the constants  in  \eqref{weakpower},
$\kappa(A^*A)= \|A^*A\|_{\mathcal B} \|(A^*A)^{-1} \|_{\mathcal B}$,
$ a= \big(1- (\kappa(A^*A))^{-1}\big)^{-1}> 1$,
$$b= \frac{ \|I\|_{\mathcal A}+   \|A^* A\|_{\mathcal B}^{-1}
 \|A^* A\|_{\mathcal A}} {1- (\kappa(A^*A))^{-1} }\ge a >1,
$$
and
\begin{equation}\label{t0.def0}
t_0=\Big( \frac{ (m-1)(m-\theta) \log_m (m-\theta) \ln (Db)}{(m-1-\theta) \ln a}\Big)^{\ln m/(\ln m-\ln (m-\theta))} \ {\rm for}\  0<\theta<m-1.
\end{equation}

\end{thm}

\begin{proof} 
Obviously it suffices to prove \eqref{norm-control1}.
  In this paper, we follow the argument in \cite{shinsun19} to give a sketch proof.
Let $A\in {\mathcal A}$ so that $A^{-1} \in {\mathcal B}$.
As $\mathcal{B}$ is a symmetric $*$-algebra,  the spectrum of $A^*A$ in ${\mathcal B}$ lies in an interval on the positive real axis,
\begin{equation} \sigma (A^* A)\subset
\big[  \|(A^* A)^{-1}\|_{\mathcal B}^{-1}, \   \|A^* A\|_{\mathcal B}\big].
\end{equation}
Therefore
 $B:=I- \|A^* A\|_{\mathcal B}^{-1} A^* A\in {\mathcal A}$ satisfies
\begin{equation}\label{B-norm}
 \|B\|_{\mathcal B} \le 1- (\kappa(A^*A))^{-1}=a^{-1}<1
\end{equation}
and
\begin{equation}\label{A-norm}
\|B\|_{\mathcal A} \le \|I\|_{\mathcal A}+ \|A^* A\|_{\mathcal B}^{-1}
 \|A^* A\|_{\mathcal A}=ba^{-1}.
\end{equation}

For a positive integer $n=\sum_{i=0}^l \varepsilon_i m^i$, define  $n_0=n$ and
$n_k, 1\le k\le l$, inductively by
\begin{equation}\label{nk.def}
n_{k}= \frac{n_{k-1}-\varepsilon_{k-1}}{m}=\sum_{i=k}^l \varepsilon_i m^{i-k}, 1\le k\le l,  \end{equation}
 where $\varepsilon_i\in\{0,1, \ldots, m-1\}, 1\le i\le l-1$ and $\varepsilon_l\in \{1, \ldots, m-1\}$.
By \eqref{weakpower} and \eqref{banachalgebra.defnew2}, we have
\begin{equation}\label{2-exponents}
\|B^{m n_k} \|_{\mathcal A} \le D \| B^{n_k}\|_{\mathcal A}^{m-\theta} \| B^{n_k}\|_{\mathcal B}^{\theta}
\le D \|B^{n_k}\|_{\mathcal A}^{m-\theta} \|B\|_{\mathcal B}^{n_k\theta},
\quad k=1,\cdots, l-1.
\end{equation}

By   \eqref{banachalgebra.defnew2}, \eqref{banachalgebra.defnew1}, \eqref{B-norm}, \eqref{A-norm}, \eqref{nk.def} and
\eqref{2-exponents}, we obtain 
\begin{eqnarray}  \label{tilde-Bn-compu}
  \|B^n\|_{\mathcal A}  & = &  \|B^{n_0}\|_{\mathcal A}
\le \|B^{m n_1}\|_{\mathcal A}
\|B\|_{\mathcal A}^{\varepsilon_0}\le  D \|B^{ n_1}\|_{\mathcal A}^{m-\theta}
\|B\|_{\mathcal A}^{\varepsilon_0}  \|B\|_{\mathcal B}^{n_1\theta}
\nonumber \\
&
\le &  D^{1+ (m-\theta)}
\|B^{n_2}\|_{\mathcal A}^{(m-\theta)^2}
\| B\|_{\mathcal A}^{\varepsilon_0+\varepsilon_1(m-\theta)}
\| B\|_{\mathcal B}^{ n_1 \theta+ n_2 \theta (m-\theta)}\nonumber\\
&  \le &  \cdots
 \nonumber \\
&
\le & D^{\sum_{k=0}^{l-1} (m-\theta)^k}
\|B\|_{\mathcal A}^{\sum_{k=0}^l \varepsilon_k (m-\theta)^k}
\| B\|_{\mathcal B}^{\theta \sum_{k=1}^l n_k (m-\theta)^{k-1}}
\nonumber \\
&
=& D^{\sum_{k=0}^{l-1} (m-\theta)^k}
\|B\|_{\mathcal A}^{\sum_{k=0}^l \varepsilon_k (m-\theta)^k}
\| B\|_{\mathcal B}^{n-\sum_{k=0}^l \varepsilon_k (m-\theta)^k}
\nonumber \\
\qquad &\le & D^{\sum_{k=0}^{l-1} (m-\theta)^k}
b^{\sum_{k=0}^l \varepsilon_k (m-\theta)^k}
a^{-n}
\nonumber \\
\qquad &
\le  & 
\left\{\begin{array}{ll} (Db)^{\frac{(m-1)(m-\theta)}{m-1-\theta} n^{\log_m(m-\theta)}}  a^{-n}  & {\rm if} \ \theta<m-1\\
 &  \\
(Db)^{(m-1)\log_m (mn+1)}  a^{-n}  & {\rm if} \ \theta=m-1,
\end{array}\right.
\end{eqnarray}
where the last inequality holds since
\begin{eqnarray*}
\sum_{k=0}^l \varepsilon_k (m-\theta)^k
& \le &  (m-1) \sum_{k=0}^l (m-\theta)^k \le (m-1) \left\{\begin{array}{ll}
 \frac{ (m-\theta)^{l+1}-1}{m-1-\theta} \  & {\rm if} \ \theta<m-1\\
 l+1 & {\rm if}  \ \theta=m-1\end{array}
 \right. \\
 & \le &
 (m-1) \left\{\begin{array}{ll}
 \frac{ m-\theta}{m-1-\theta} n^{\log_m(m-\theta)} \  & {\rm if} \ \theta<m-1\\
\log_m (mn+1) & {\rm if}  \ \theta=m-1.\end{array}\right.
\end{eqnarray*}
Observe that
 $A^*A =  \|A^* A\|_{\mathcal B} (I-B)$. Hence
$$
 A^{-1}=(A^*A)^{-1}A^* = \|A^* A\|_{\mathcal B}^{-1} \Bigg( \sum_{n=0}^\infty B^n \Bigg) A^*.
$$
This together with  \eqref{banachalgebra.defnew1}, \eqref{tilde-Bn-compu}
and \eqref{inverseA.power} implies that
\begin{eqnarray}\label{inverseA.power}
  \|A^{-1}\|_{\mathcal A}   & \le &    \|A^* A\|_{\mathcal B}^{-1}  \|A^*\|_{\mathcal A}
\sum_{n=0}^\infty \|B^n\|_{\mathcal A}
\nonumber\\
&\le &  \|A^* A\|_{\mathcal B}^{-1}  \|A^*\|_{\mathcal A}\times
\left\{\begin{array}{ll} \sum_{n=0}^\infty (Db)^{\frac{(m-1)(m-\theta)}{m-1-\theta} n^{\log_m(m-\theta)}}  a^{-n}  &   {\rm if} \ \theta<m-1\\
 \sum_{n=0}^\infty (Db)^{(m-1)\log_m (mn+1)}  a^{-n}  &  {\rm if} \ \theta=m-1.
\end{array}\right.
\end{eqnarray}

By direct calculation, we have 
\begin{eqnarray}\label{m-1.estimate}
 \sum_{n=0}^\infty (Db)^{(m-1)\log_m (mn+1)}  a^{-n} &  \le &   a \sum_{n=0}^\infty \int_n^{n+1} (Db)^{(m-1) \log_m (mt+1)} a^{-t} dt\nonumber\\
& \le & a^2 (Db)^{m-1} \int_0^\infty  (t+1)^{(m-1) \log_m (Db)} e^{-(t+1) \ln a} dt\nonumber\\
& \le & a^2 (Db)^{m-1} (\ln a)^{-1} \Gamma\Big(\frac{(m-1) \ln (Db)}{\ln m \ln a}+1\Big).
\end{eqnarray}
This together with  \eqref{inverseA.power} proves \eqref{norm-control1} for $\theta=m-1$.

For $0<\theta<m-1$,  set
$$s(t)=t- \frac{ (m-1)(m-\theta) \ln (Db)}{(m-1-\theta) \ln a} t^{\log_m(m-\theta)}.$$
Observe that
$$
s'(t)=1-  \frac{ (m-1)(m-\theta) \ln (Db)}{(m-1-\theta) \ln a} \log_m(m-\theta) t^{\log_m (1-\theta/m)}.
$$
Therefore
\begin{equation}\label{st.eq1}
\min_{t\ge 0} s(t)=s(t_0)=-\frac{\ln m-\ln (m-\theta)} {\ln (m-\theta)} t_0<0
\end{equation}
and
\begin{equation}\label{st.eq2}
1\ge s'(t)\ge  s'(2t_0)=1-  2^{\log_m (1-\theta/m)} \ \quad \ {\rm for \ all}\ \ t\ge 2t_0,
\end{equation}
where  $t_0$ is given in \eqref {t0.def0}.
By \eqref{st.eq1} and \eqref{st.eq2}, we have
\begin{eqnarray}
 & & \sum_{n=0}^\infty (Db)^{\frac{(m-1)(m-\theta)}{m-1-\theta} n^{\log_m(m-\theta)}}  a^{-n}
  \le   a \sum_{n=0}^\infty \int_n^{n+1}  (Db)^{\frac{(m-1)(m-\theta)}{m-1-\theta} t^{\log_m(m-\theta)}}  a^{-t} dt\nonumber\\
& = & a\Big(\int_0^{2t_0}+\int_{2t_0}^\infty\Big)
\exp(- s(t) \ln a)  dt\nonumber\\
& \le &  2at_0
 \exp(-s(t_0) \ln a) +    (1-  2^{\log_m (1-\theta/m)})^{-1} a \int_{s(2t_0)}^\infty \exp(-u \ln a) du \nonumber\\
 & \le & \Big  (2t_0+(1-  2^{\log_m (1-\theta/m)})^{-1} (\ln a)^{-1}\Big )  a
 \exp\Big(\frac{\ln m-\ln (m-\theta)} {\ln (m-\theta)} t_0 \ln a\Big) .
\end{eqnarray}
Combining the above estimate with  \eqref{inverseA.power} proves \eqref{norm-control1} for $\theta<m-1$.
\end{proof}

For $m=2$, the estimate \eqref{norm-control1} to the inverse $A^{-1}\in {\mathcal A}$  is essentially established in \cite{gkII, gkI} for $\theta=1$ and \cite{shinsun19, sunaicm14} for  $\theta\in (0, 1)$, and a similar estimate is given  in \cite{ samei19}. The reader may refer to \cite{fangshinsun13, grochenigklotz10, sunacha08,  suntams07,
suncasp05} for  norm estimation of the inverse of elements in Banach algebras of infinite matrices and integral operators with certain off-diagonal decay.

\begin{rem}{\rm  A good estimate for the norm control function $h$ in
the norm-controlled inversion \eqref{normcontrol} is important for some mathematical and engineering applications.
For an element $A\in {\mathcal A}$ with $A^{-1}\in {\mathcal B}$, we obtain the following estimate from Theorem \ref{main-thm1}:
\begin{equation}
\|(A^*A)^{-1}\|_{\mathcal A}
 \le
C \|A^* A\|_{\mathcal B}^{-1}  a  (\ln a)^{-1}  \times \left\{\begin{array}{ll}
  t_1  \exp(C t_1)
  & {\rm if} \ \theta<m-1\\
 a  b^{m-1} \exp\Big(C \frac{\ln b}{\ln a} \ln \big(\frac{\ln b}{\ln a}\big)\Big)
&   {\rm if} \ \theta=m-1,
\end{array}\right.
\end{equation}
where  $C$ is an absolute constant independent of $A$ and
$$t_1= (\ln b)^{\ln m/(\ln m-\ln (m-\theta))} (\ln a)^{-\ln (m-\theta)/(\ln m-\ln (m-\theta))}.$$
We remark that the above norm estimate to the inversion is far away from the optimal estimation for our illustrative differential subalgebra
$C^1[a, b]$. In fact,  give any $f\in C^1[a, b]$ being invertible in  $C[a, b]$,
we have
\begin{equation*}
\|1/f\|_{C^1[a, b]}\le
 \| f'\|_{C[a, b]} \|f^{-1} \|_{C[a, b]}^2+ \|1/f \|_{C[a, b]} \le  \|1/f \|_{C[a, b]}^2 \|f\|_{C^1[a, b]}.
 \end{equation*}
 Therefore  $C^1[a, b]$ has norm-controlled inversion in $C[a, b]$ with the norm control function
 $h(s, t)$ in \eqref{normcontrol} being
 $h(s, t)= s t^2$.
Gr\"ochenig and Klotz first considered  norm-controlled inversion
with the norm control function $h$ having polynomial growth, and they show in
\cite{gkII} that the Baskakov-Gohberg-Sj\"ostrand  algebra ${\mathcal C}_{1, \alpha}, \alpha>0$
and  the Jaffard algebra ${\mathcal J}_\alpha, \alpha>1$  have norm-controlled inversion in ${\mathcal B}(\ell^2)$ with
 the norm control function
 $h$ bounded by a polynomial. In \cite{shinsun19}, we  
proved that  the  Beurling algebras  ${\mathcal B}_{p, \alpha}$
 with $1\le p\le \infty$ and $\alpha>1-1/p$ admit  norm-controlled inversion in ${\mathcal B}(\ell^2)$ with
 the norm control function
 bounded by some polynomials.  Following the commutator technique used in \cite{shinsun19, sjostrand94}, we can establish a similar result for
 the Baskakov-Gohberg-Sj\"ostrand  algebras ${\mathcal C}_{p, \alpha}$ with $1\le p\le \infty$ and $\alpha>1-1/p$.

 \begin{thm}\label{polynomial.thm}
 Let $1\le p\le \infty$ and $\alpha>1-1/p$. Then
  the Baskakov-Gohberg-Sj\"ostrand  algebra ${\mathcal C}_{p, \alpha}$
  and the  Beurling algebra  ${\mathcal B}_{p, \alpha}$
  admit  norm-controlled inversion in ${\mathcal B}(\ell^2)$ with
 the norm control function
 bounded by a polynomial.
   \end{thm}

 It is still unknown whether   Gr\"ochenig-Schur algebras ${\mathcal A}_{p, \alpha}, 1\le p<\infty, \alpha>1-1/p$,
  admit norm-controlled inversion in ${\mathcal B}(\ell^q), 1\le q< \infty$, with
 the norm control function  bounded by a polynomial.
In  \cite{gkII}, Gr\"ochenig and Klotz introduce a differential operator  ${\mathcal D}$ on a Banach algebra and
use it to define a differential $*$-algebra  ${\mathcal A}$ of a symmetric $*$-algebra  ${\mathcal B}$, which
  admits  norm-controlled inversion with
 the norm control function
 bounded by a polynomial.  However, the differential algebra in \cite{gkII} does not include
 the  Gr\"ochenig-Schur algebras ${\mathcal A}_{p, \alpha}$,
 the Baskakov-Gohberg-Sj\"ostrand  algebra ${\mathcal C}_{p, \alpha}$
  and the  Beurling algebra  ${\mathcal B}_{p, \alpha}$ with $1\le p<\infty$ and $\alpha>1-1/p$.
   It could be an interesting problem to  extend the conclusions in Theorem \ref{polynomial.thm}
   to general Banach algebras such that
 the  norm control functions in the norm-controlled inversion have polynomial growth.
}
\end{rem}

\begin{rem} {\rm  A crucial step in the
proof of Theorem \ref{main-thm1} is to introduce
  $B:=I- \|A^* A\|_{\mathcal B}^{-1} A^* A\in {\mathcal A}$, whose spectrum is contained in an interval
   on the positive real axis.
   The above reduction  depends on the requirements that
   ${\mathcal B}$ is symmetric and both ${\mathcal A}$ and ${\mathcal B}$ are $*$-algebras with common identity and involution.
For the applications to some mathematical and
engineering fields, the  widely-used algebras ${\mathcal B}$ of infinite matrices and integral operators
are  the operator algebras ${\mathcal B}(\ell^p)$ and  ${\mathcal B}(L^p), 1\le p\le \infty$, which are symmetric only when $p=2$.
In \cite{akramjfa09,  fangshinsun13, shinsun19, shincjfa09,  sunacha08, tesserajfa10}, inverse-closedness of localized matrices and integral operators
in ${\mathcal B}(\ell^p)$ and ${\mathcal B}(L^p), 1\le p\le \infty$, are discussed, and in \cite{fangshinsun20},
Beurling algebras  ${\mathcal B}_{p,\alpha}$ with $1\le p<\infty$ and $\alpha>d(1-1/p)$ are shown to admit polynomial norm-controlled
inversion in nonsymmetric algebras ${\mathcal B}(\ell^p), 1\le p<\infty$.
It is  still widely  open to discuss  Wiener's lemma and
 norm-controlled inversion when  ${\mathcal B}$ and ${\mathcal A}$ are not $*$-algebras and
 ${\mathcal B}$ is not a symmetric  algebra.
}
\end{rem}

\end{document}